\documentclass[smallextended,numbook,runningheads]{svjour3}     
\smartqed  
\usepackage{graphicx}
\usepackage{amsmath}
\usepackage{mathptmx}      
\journalname{BIT}
\begin{document}

\title{On the Convergence of  Ritz Pairs and Refined  Ritz Vectors for
      Quadratic Eigenvalue Problems\thanks{The first and third authors were supported
      in part by the National Science Council, the National Center for Theoretical
      Sciences, the Center of Mathematical Modeling and Scientific Computing, and the Chiao-Da ST Yau Center  in
      Taiwan, and the second author was supported in part by
National Basic Research Program of China  2011CB302400 and National
Science Foundation of China (No. 11071140).
}}


\author{Tsung-Ming Huang       \and
        Zhongxiao Jia \and
        Wen-Wei Lin
}


\institute{T.-M Huang \at
              Department of
    Mathematics, National Taiwan Normal University, Taipei 116,
    Taiwan     \email{min@ntnu.edu.tw}           
           \and
           Z. Jia \at
           Corresponding author.  Department of Mathematical Sciences, Tsinghua University,
Beijing 100084, China
\email{jiazx@tsinghua.edu.cn}
           \and
           W.-W Lin \at
Department of Applied Mathematics, National Chiao Tung University, Hsinchu 300, Taiwan
\email{wwlin@math.nctu.edu.tw}
}

\date{Received: date / Accepted: date}

\maketitle

\begin{abstract}
For a given subspace, the Rayleigh-Ritz method projects
the large quadratic eigenvalue problem (QEP) onto it and produces
a small sized dense QEP. Similar to the Rayleigh-Ritz method for the
linear eigenvalue problem, the  Rayleigh-Ritz method defines the  Ritz
values and the  Ritz vectors of the QEP with respect to the projection
subspace. We analyze the convergence of the method when
the angle between the subspace and the desired eigenvector
converges to zero. We prove that there is a  Ritz value that converges
to the desired eigenvalue unconditionally but the  Ritz vector converges
conditionally and may fail to converge. To remedy the drawback of possible
non-convergence of the  Ritz vector, we propose a refined  Ritz vector
that is mathematically different from the  Ritz vector and is proved to
converge unconditionally. We construct examples to illustrate our theory.


\end{abstract}





\numberwithin{equation}{section}

\newcommand{\abs}[1]{\lvert#1\rvert}

\newcommand{\blankbox}[2]{%
  \parbox{\columnwidth}{\centering
    \setlength{\fboxsep}{0pt}%
    \fbox{\raisebox{0pt}[#2]{\hspace{#1}}}%
  }%
}

\section{Introduction}
\label{intro}

Consider the numerical solution of the large quadratic eigenvalue
problem (QEP)
\begin{eqnarray}
     \mathcal{Q}(\lambda) x \equiv ( \lambda^2 M + \lambda D + K ) x = 0,
     \label{eq:QEP}
\end{eqnarray}
where $\lambda \in \mathcal{C}$, $x \in \mathcal{C}^{n} \backslash \{0\}$,
$M$, $D$ and $K$ are $n \times n$ complex matrices with $M=M^H>0$ Hermitian
positive definite. The scalar $\lambda$ and the nonzero vector $x$
in \eqref{eq:QEP} are called an eigenvalue and a corresponding eigenvector
of the quadratic pencil $\mathcal{Q}(\lambda)$  or $( M,D,K )$, respectively.
The pair $(\lambda, x)$ is called an eigenpair of $( M,D,K )$. Since
$M=M^H>0$ in \eqref{eq:QEP}, $\mathcal{Q}(\lambda)$ has $2n$ finite eigenvalues.

QEP \eqref{eq:QEP} arises in a wide variety of scientific and engineering
applications \cite{BetHigMehSchTis_NLEVP,TissMeer01}.
The theoretical framework for general
matrix polynomials and in particular for quadratic pencils can be
found in books by Lancaster  \cite{Lancaster_66} and more recently by Gohberg,
Lancaster and Rodman \cite{GohLanRod_82}. A good survey of
mathematical properties, perturbation analysis, and a variety of numerical
algorithms for QEPs can be found in the paper by Tisseur and
Meerbergen \cite{TissMeer01}.

In practice, a small number of eigenvalues that are nearest to a target $\tau$
or located in a prescribed region of the complex plane
and the corresponding eigenvectors are often of interest.
To this end, we exploit the shift transformation
$\lambda_{\tau} = \lambda - \tau$ with $\mathrm{det}(\mathcal{Q}(\tau)) \neq 0$
to transform \eqref{eq:QEP} to a new QEP of the form
\begin{eqnarray}
     \mathcal{Q}_{\tau}(\lambda_{\tau})x \equiv (\lambda_{\tau}^2 M_{\tau} +
     \lambda_{\tau} D_{\tau} + K_{\tau})x=0, \label{eq:shift-and-invert QEP}
\end{eqnarray}
where $M_{\tau} = M$, $D_{\tau} = 2 \tau M + D$ and $K_{\tau} = \tau^2 M
+ \tau D + K$ is nonsingular. So, without loss of generality,
throughout the paper, we assume that the eigenvalues to be sought are nonzero.

One kind of classical methods for solving QEP \eqref{eq:QEP} is to reformulate
it as a certain standard (or generalized) eigenvalue problem via a so-called
linearization process and then to apply Krylov subspace based methods
or Jacobi-Davidson type methods to solve the corresponding linear
eigenvalue problem. Most of these methods fall into
the category of the Rayleigh-Ritz method that
is widely used for the computation of partial eigenpairs of a standard
linear eigenvalue problem from a given projection subspace. As is well known,
under the assumption that the angle between a desired eigenvector and
the projection subspace tends to zero, there exists a Ritz value that
converges to the desired eigenvalue unconditionally but its
corresponding Ritz vector may fail to converge; furthermore, when one
is concerned with eigenvectors, one can compute certain refined Ritz vectors
whose convergence is guaranteed
\cite{Jia_97,Jia_02,Jia_05,JiaStewart_99,JiaStewart_01}; see also
\cite{Stewart_MaxAlg_Eig_01}.

Over the years, some reliable numerical methods have been
proposed that are used to solve large and sparse QEPs directly. Based on
certain orthogonal projection conditions, various methods are designed to construct
suitable lower dimensional subspaces. Then, the large QEP is projected onto a
given subspace to produce a small sized dense QEP which can be solved by the
standard QR or QZ algorithm. They fall into the category of
the  Rayleigh-Ritz method, as will be described in the next paragraph.
Methods of this type include the residual inverse iteration method
\cite{Huitfeldt-Ruhe90,Meerbergen01,Neumaier85},
the Jacobi-Davidson method \cite{SBFV:96,Sleijpen-vanderVorst-vanGijzen96},
Krylov subspace type methods \cite{Hoffnung-Li-Ye06,Li-Ye03}, the nonlinear
Arnoldi method \cite{Voss04}, second-order Arnoldi (SOAR) type methods
\cite{Bai-Su05,JiaSun11,Lin-Bao06,Wang-Su-Bai06}, the iterated shift-and-invert
Arnoldi method \cite{Ye06} and the semiorthogonal generalized Arnoldi (SGA)
method \cite{Huang-Li-Li-Lin11}.

Now we describe the  Rayleigh-Ritz method for the QEP.
For a given orthonormal matrix $Q \in \mathcal{C}^{n \times m}\ (m \leq n)$,
the  Rayleigh-Ritz method is to find a scalar $\mu \in \mathcal{C}$ and a
unit length vector $\hat{x} \in \mathcal{C}^{m}$ satisfying the orthogonal projection
condition
\begin{align*}
    (\mu^{2} MQ + \mu DQ + KQ)\hat{x} \ \perp\ \mbox{span}\{Q\},
\end{align*}%
which amounts to solving the projected QEP
\begin{equation}
( \mu^2 \widehat{M} +
\mu \widehat{D} + \widehat{K} ) \hat{x}=0, \label{pqep}
\end{equation}
where
\begin{eqnarray}
    \widehat{M} = Q^H M Q, \quad  \widehat{D} = Q^H D Q, \quad \widehat{K}
    = Q^H K Q. \label{eq:q_RR triple}
\end{eqnarray}
If $(\mu, \hat{x})$ with $\|\hat{x}\|=1$ is an eigenpair of $( \widehat{M},
\widehat{D},\widehat{K} )$, i.e., $ (\mu^{2}\widehat{M} + \mu \widehat{D} +
\widehat{K})\hat{x} = 0$, then $\mu$ and $Q\hat{x}$ are, respectively,
called a  Ritz value  and a corresponding  Ritz vector of $( M,D,K )$
with respect to ${\rm span}\{Q\}$, and $(\mu, Q \hat{x})$ is a
 Ritz pair of $( M,D,K )$. Since $M$ is Hermitian positive definite, so
is $\widehat{M}$ for any given $Q$.  Therefore, we have $2m$ finite  Ritz
values.

For a {\em given} $Q$, the assumption that $M$ is Hermitian
positive definite is a {\em sufficient} condition to ensure the finiteness of
both the eigenvalues and the Ritz values.
Without this assumption, $\widehat{M}$
would possibly be {\em singular} for some given orthonormal $Q$. In this case,
there could be some {\em infinite}  Ritz values, the situation would
become much more complicated, and the  Rayleigh--Ritz method may fail to work.
Indeed, as will be seen, some of our important convergence conclusions
cannot be drawn, e.g., the bound in Theorem~\ref{cor:qritzvalue} may not tend to zero
when the subspace $\mbox{span}\{Q\}$ is sufficiently good.
In contrast, as will be clear, QEP \eqref{eq:QEP}
is mathematically equivalent to some standard linear eigenvalue problem
provided that $M$ is nonsingular; see
\eqref{eq:1stGEP}. It is well known that
the standard Rayleigh--Ritz method for the {\em linear} eigenvalue problem
{\em always} computes {\em finite} Ritz values for {\em any} projection subspace.
Therefore, there are some essential differences between the Rayleigh--Ritz
method for (\ref{eq:QEP}) and the method for the linear eigenvalue problem.
As is expected, it is nontrivial to establish a convergence theory
of the Rayleigh--Ritz method for (\ref{eq:QEP}).
As a key step of our further discussions,
we first assume the finiteness of Ritz values for {\em any} projection
subspace span$\{Q\}$. It is simple to justify that for {\em any} orthonormal
$Q$ the Hermitian positive definiteness of $M$ is sufficient to ensure
that of $\widehat{M}$. Generally, what we need in the paper is to
assume that $\|\widehat{M}^{-1}\|$ is uniformly bounded independently of $Q$.
This assumption is true if $M$ is Hermitian positive definite,
as $\|\widehat{M}^{-1}\|\leq\|M^{-1}\|$ for any orthonormal $Q$.
So, purely for simplicity of presentation, we assume that $M$ is
Hermitian positive definite throughout the paper. Nevertheless, we must
keep it in mind that all the convergence results and claims are true
in this paper provided that $\widehat{M}$ is nonsingular and
$\|\widehat{M}^{-1}\|$ is bounded.

In this paper we study the convergence of the  Ritz
value and the corresponding  Ritz vector, and extend some of the results
in \cite{JiaStewart_99,JiaStewart_01,Stewart_MaxAlg_Eig_01} to the
Rayleigh-Ritz method for (\ref{eq:QEP}). Although a number of
Rayleigh-Ritz procedures with respect to different subspaces have been used,
to our best knowledge, there has been no unified convergence result
and general theory. As will be seen later,
carrying out this task is indeed nontrivial and complicated. We establish some
important results similar to those for the linear
eigenvalue problem. It turns out that there exists a  Ritz value that
converges to the desired eigenvalue unconditionally but
the corresponding  Ritz vector may fail to converge even if the corresponding
projection subspace ${\rm span}\{Q\}$ contains a sufficiently accurate
approximation to the desired eigenvector.
It is thus necessary and significant to replace the
 Ritz vector by a refined  Ritz vector that has residual minimization and
is mathematically different from the  Ritz vector.
We prove that the refined  Ritz vector converges
unconditionally provided that the angles between the desired eigenvector
and the subspaces tend to zero.  All convergence results are nontrivial
generalizations of the known results on the Rayleigh-Ritz method and the
refined Rayleigh--Ritz method for the linear eigenvalue problem in
\cite{JiaStewart_99,JiaStewart_01,Stewart_MaxAlg_Eig_01}.

This paper is organized as follows. In Section~\ref{sec:conv_q_Ritz},
we analyze the convergence for  Ritz values and  Ritz vectors and prove that
the  Ritz value is unconditionally convergent but the associated  Ritz vector
may fail to converge. To remedy this drawback, in
Section~\ref{sec:conv_refined_q_Ritz}, we introduce  a refined  Ritz vector
and prove its unconditional convergence.
Finally, we conclude the paper in Section~\ref{sec:conclusion}.

Throughout this paper, the superscripts $H$ and $T$
denote the conjugate transpose and the transpose of a matrix or vector,
respectively.  $I_n$ is the identity matrix of order $n$. We denote by
$\| \cdot \|$ both Euclidean vector norm and the spectral matrix norm.

\section{Convergence of  Ritz values and  Ritz vectors} \label{sec:conv_q_Ritz}

Throughout the paper, let $(\lambda_1, x_1)$ with $\| x_1 \| = 1$
be a desired eigenpair of $( M,D,K )$ and assume that $\lambda_1$ is simple.
Furthermore, we keep in mind the assumption made in the introduction
that $\lambda_1\neq 0$, which is without loss of generality due to the
equivalence of (\ref{eq:QEP}) and (\ref{eq:shift-and-invert QEP}).

We convert QEP \eqref{eq:QEP} to a generalized eigenvalue problem (GEP)
of the form%
\begin{subequations}\label{eq:1stGEP}
\begin{align}\label{eq:1stGEP_eq}
A \left[ \begin{array}{c}
         \lambda x \\ x
     \end{array} \right] = \lambda B \left[ \begin{array}{c}
         \lambda x \\ x
     \end{array} \right],
\end{align}
or a standard linear eigenvalue problem (LEP) of the form
\begin{align}\label{eq:1stSEP_eq}
B^{-1}A \left[ \begin{array}{c}
         \lambda x \\ x
     \end{array} \right] = \lambda \left[ \begin{array}{c}
         \lambda x \\ x
     \end{array} \right],
\end{align}
where
\begin{align}
A = \left[ \begin{array}{cc} -D & -K \\ I_n & 0 \end{array} \right]\quad \text{and}\quad
B = \left[ \begin{array}{cc} M & 0 \\ 0 & I_n \end{array} \right]. \label{eq:mtx_A_B_1stGEP}
\end{align}
\end{subequations}
So $\lambda_1$ is an eigenvalue of the matrix pencil $( A,B )$ or the matrix
$B^{-1}A$ in \eqref{eq:1stGEP} and $v_1  \equiv \left[ \begin{array}{c}
\lambda_1 x_1 \\ x_1\end{array} \right] / \sqrt{1+|\lambda_1|^2 }$
is its corresponding normalized eigenvector. There
are numerous linearizations of QEP \eqref{eq:QEP}.
We use \eqref{eq:1stGEP} for two reasons. The first is that it is a
very commonly used linearlization in the literature. The second is that we
establish our results in this paper by relating the QEP to such linearization.
Other linearizations are certainly possible and useable, but if then we may
have to make a very different and more complicated
analysis in order to establish the convergence theory of
the Rayleigh-Ritz method and refined Ritz vectors for the QEP.

There are unitary matrices
$[v_1, \ X]$ and $[y_1,\ Y] \in \mathcal{C}^{2n\times 2n}$
with $v_1, y_1 \in \mathcal{C}^{2n}$ such that
\begin{align}
     \begin{bmatrix}
          y_1^H \\ Y^H
     \end{bmatrix} A \begin{bmatrix}
          v_1 & X
     \end{bmatrix} = \begin{bmatrix}
          \alpha & s^H \\ 0 & L
     \end{bmatrix}, \quad \begin{bmatrix}
          y_1^H \\ Y^H
     \end{bmatrix} B \begin{bmatrix}
          v_1 & X
     \end{bmatrix} = \begin{bmatrix}
          \beta & t^H \\ 0 & N
     \end{bmatrix}, \label{eq:RR_AB}
\end{align}%
where  $L, N \in \mathcal{C}^{(2n-1) \times (2n-1)}$ and $\lambda_1 = \alpha \beta^{-1}$.
Since $\lambda_1$ is supposed to be simple, it is not an eigenvalue of $( L,N )$.

For a given orthonormal matrix $Q \in \mathcal{C}^{n \times m}$ with $m \leq n$, define
      \begin{eqnarray}
      W = \left[ \begin{array}{cc} Q & 0 \\ 0 & Q \end{array} \right] \label{eq:mtx_W}
      \end{eqnarray}
and let $[ Q, Q^{\bot}]$ be unitary
with $Q^{\bot} \in \mathcal{C}^{n \times (n-m)}$. From now on, throughout the paper,
let $\theta_1$ be the acute angle between $x_1$ and the projection subspace
${\rm span}\{Q\}$ and
\begin{eqnarray}
       q_1 = Q^H x_1, \quad q_1^{\bot} = ( Q^{\bot} )^H x_1. \label{eq:theta}
\end{eqnarray}
Then it holds that \cite[p. 249, Theorem 2.2]{Stewart_MaxAlg_Eig_01}
\begin{eqnarray}
      \| q_1^{\bot} \| = \sin \theta_1, \quad \| q_1 \| = \sqrt{1-\sin^2\theta_1}
      = \cos \theta_1. \label{eq:q1_norm}
\end{eqnarray}
First of all, we want to show that there is a Ritz value $\mu_1$ that converges
to $\lambda_1$ unconditionally when $\sin \theta_1 \to 0$.
The following perturbation result is needed, which is expressed in terms of
the a priori uncomputable $\tan\theta_1$ and is different from
Theorem 1 in \cite{tisseur00}, which is a backward perturbation result
in terms of the a posteriori
computable residual norm of an approximate eigenpair.

\begin{lemma} \label{thm:conv_thm1}
     With $\lambda_1,\ q_1$ and $\theta_1$ defined as above.
     Let $\widehat{M}, \widehat{D}$ and $\widehat{K}$ be defined
     in \eqref{eq:q_RR triple} and $\widehat{q_1}=q_1/\|q_1\|$.
     Then there are perturbation matrices $\mathcal{E}_{\widehat{M}},
     \mathcal{E}_{\widehat{D}}, \mathcal{E}_{\widehat{K}}\in
     \mathcal{C}^{m\times m}$ with
     \begin{subequations} \label{eq:lowerbdd_eps}
           \begin{eqnarray}
                 && \| \mathcal{E}_{\widehat{M}} \| \leq \frac{1}{3} \left( m_0 +
                 \frac{1}{| \lambda_1 |} d_0 +
                 \frac{1}{| \lambda_1 |^2} k_0 \right) \tan \theta_1,
                 \label{eq:lowerbdd_eM} \\
                 && \| \mathcal{E}_{\widehat{D}} \| \leq \frac{1}{3}
                 \left( | \lambda_1 | m_0 +  d_0 +
                 \frac{1}{| \lambda_1 |} k_0 \right) \tan \theta_1,
                 \label{eq:lowerbdd_eD} \\
                 && \| \mathcal{E}_{\widehat{K}} \| \leq \frac{1}{3}
                 \left( |\lambda_1 |^2 m_0 +
                 | \lambda_1 | d_0 +  k_0 \right) \tan \theta_1,
                 \label{eq:lowerbdd_eK}
           \end{eqnarray}
     \end{subequations}
     such that $(\lambda_1,\widehat{q}_1)$ is an exact eigenpair of the perturbed
     $( \widehat{M} + \mathcal{E}_{\widehat{M}}, \widehat{D} + \mathcal{E}_{\widehat{D}},
     \widehat{K} + \mathcal{E}_{\widehat{K}} )$, where
     \begin{equation}\label{m0d0k0}
     m_0=\|M\|,\ d_0=\|D\|,\ k_0=\|K\|.
     \end{equation}
\end{lemma}
\begin{proof}
      Recalling \eqref{eq:theta} and \eqref{eq:q1_norm},
      since
      \begin{align*}
            0 = \left( \lambda_1^2 M + \lambda_1 D + K \right) x_1 = \left( \lambda_1^2 M +
            \lambda_1 D + K \right) \left[ \begin{array}{cc}
                   Q & Q^{\bot}
            \end{array} \right] \left[ \begin{array}{c}
                   Q^H \\ ( Q^{\bot})^H
            \end{array} \right] x_1, 
      \end{align*}
      we obtain
      \begin{eqnarray}
           \lambda_1^2 M Q q_1 + \lambda_1 D Q q_1 + K Q q_1 = - \left( \lambda_1^2 M +
           \lambda_1 D+ K \right) Q^{\bot} q_1^{\bot}. \label{eq:pf_conv_thm1_2}
      \end{eqnarray}
      Pre-multiplying \eqref{eq:pf_conv_thm1_2} by $Q^H$ gives
      \begin{equation}
            r_1 \equiv ( \lambda_1^2 \widehat{M} + \lambda_1 \widehat{D} + \widehat{K} )
            \widehat{q}_1 = - ( \lambda_1^2 Q^H M + \lambda_1 Q^H D + Q^H K ) Q^{\bot}
            \frac{q_1^{\bot}}{\| q_1 \|}. \label{eq:pf_conv_thm1_3}
      \end{equation}
      So, noting from (\ref{eq:q1_norm}) that $\tan\theta_1=\frac{\sin\theta_1}
      {\cos\theta_1}=\frac{\|q_1^{\bot}\|}{\| q_1 \|}$, we have
      $$
      \|r_1\|\leq (|\lambda_1|^2m_0+|\lambda_1|d_0+k_0)\tan\theta_1.
      $$
      Define
      \begin{align*}
            \mathcal{E}_{\widehat{M}} = - \frac{1}{3 \lambda_1^2} r_1 \widehat{q}_1^H, \quad
            \mathcal{E}_{\widehat{D}} = -\frac{1}{3\lambda_1} r_1 \widehat{q}_1^H, \quad
            \mathcal{E}_{\widehat{K}} = - \frac{1}{3} r_1 \widehat{q}_1^H.
      \end{align*}
      By \eqref{eq:pf_conv_thm1_3} it is easily seen that $\| \mathcal{E}_{\widehat{M}} \|$,
      $\| \mathcal{E}_{\widehat{D}} \|$ and $\| \mathcal{E}_{\widehat{K}} \|$
      satisfy \eqref{eq:lowerbdd_eps} and
      \begin{eqnarray*}
           \left[\ \lambda_1^2 ( \widehat{M} + \mathcal{E}_{\widehat{M}} ) + \lambda_1
           ( \widehat{D} + \mathcal{E}_{\widehat{D}} ) + ( \widehat{K} +
           \mathcal{E}_{\widehat{K}} ) \ \right] \widehat{q}_1 = 0,
      \end{eqnarray*}
      which completes the proof.
\end{proof}

We may deduce from this lemma that there exists an eigenvalue $\mu_1$ of
$(\widehat{M},\widehat{D},\widehat{K})$ that converges to $\lambda_1$ as
$\theta_1\rightarrow 0$. However, things are subtle and by no means trivial
here. The difficulty is that, unlike a usual matrix perturbation problem where
matrices are {\em given} and {\em fixed} and perturbations are allowed to
{\em change}, here the matrix triple $(\widehat{M},\widehat{D},\widehat{K})$
and the perturbation triple $(\mathcal{E}_{\widehat{M}},\mathcal{E}_{\widehat{D}},
\mathcal{E}_{\widehat{K}})$ {\em change simultaneously} as $\theta_1\rightarrow 0$.
This means that there may be a
possibility that, as $\theta_1$ changes, the eigenvalue $\lambda_1$ of
$(\widehat{M} +\mathcal{E}_{\widehat{M}},\widehat{D}+\mathcal{E}_{\widehat{D}},
\widehat{K}+\mathcal{E}_{\widehat{K}})$ and
the eigenvalues of $(\widehat{M},\widehat{D},\widehat{K})$
become ill conditioned so swiftly
that no eigenvalue of $(\widehat{M},\widehat{D},\widehat{K})$ converges to
$\lambda_1$ though $\theta_1\rightarrow 0$.

Fortunately, by exploiting a theorem of Elsner \cite{elsner82}
(also see \cite[p.168]{StewartSun90})
we can prove that this cannot happen and there is indeed an eigenvalue $\mu_1$
that converges to the desired $\lambda_1$ provided that $\theta_1\rightarrow 0$.
Elsner's theorem states that, given matrices $C$ and $\tilde{C}$ of order $n$,
for any eigenvalue $\lambda$ of $C$ there is an eigenvalue $\tilde{\lambda}$ of
$\tilde{C}$ such that
$$
|\lambda-\tilde{\lambda} |\leq (\|C\|+\|\tilde{C}\|)^{1-\frac{1}{n}}
\|C-\tilde{C}\|^{\frac{1}{n}}.
$$
For our purpose, define the matrices $\widehat{A}$ and $\widehat{B}$ by
\begin{eqnarray}
     \widehat{A} = \left[ \begin{array}{cc}
           -\widehat{D} & -\widehat{K} \\ I_m & 0
     \end{array} \right], \quad \widehat{B} = \left[ \begin{array}{cc}
           \widehat{M} & 0 \\ 0 & I_m
     \end{array} \right].  \label{eq:mtx_hat_AB}
\end{eqnarray}
Then the eigenvalues $\mu$ of $(\widehat{M},\widehat{D},\widehat{K})$ are equal to
those of $( \widehat{A}, \widehat{B} )$, whose normalized eigenvectors $\hat{v}
\equiv \left[ \begin{array}{c} \mu \hat{x} \\ \hat{x} \end{array}
\right]/\sqrt{1+| \mu|^2}$ with $\hat{x}$ the eigenvectors
associated with the eigenvalues $\mu$ of $(\widehat{M},\widehat{D},\widehat{K})$.
Since $\widehat{M}$ is Hermitian positive definite, so is $\widehat{B}$.
Therefore, all the $\mu$ are the eigenvalues of $\widehat{B}^{-1}\widehat{A}$.
Furthermore, it holds that $\|\widehat{B}^{-1}\|\leq \|B^{-1}\|$
for any given orthonormal $Q$ and Hermitian positive definite $M$.

From Lemma~\ref{thm:conv_thm1},
$\lambda_1$ is an eigenvalue of $(\widehat{A}+\mathcal{E}_{\widehat{A}},
\widehat{B}+\mathcal{E}_{\widehat{B}})$
with the perturbation matrices
\begin{eqnarray*}
     \mathcal{E}_{\widehat{A}} = \left[ \begin{array}{cc}
           -\mathcal{E}_{\widehat{D}} & -\mathcal{E}_{\widehat{K}} \\ 0& 0
     \end{array} \right], \quad \mathcal{E}_{\widehat{B}} =
     \left[ \begin{array}{cc}
           \mathcal{E}_{\widehat{M}} & 0 \\ 0 & 0
     \end{array} \right],
\end{eqnarray*}
i.e., an eigenvalue of $(\widehat{B}+\mathcal{E}_{\widehat{B}})^{-1}
(\widehat{A}+\mathcal{E}_{\widehat{A}})$ if $(\widehat{B}+
\mathcal{E}_{\widehat{B}})^{-1}$ exists. Since $\widehat{B}$ is
Hermitian positive definite and its smallest singular value is bounded by
that of $B$ from below, $ \widehat{B}+\mathcal{E}_{\widehat{B}}$ must be
nonsingular for $\theta_1$ small enough. Moreover,
for $\theta_1\rightarrow 0$, it follows from Lemma~\ref{thm:conv_thm1} that
\begin{equation}\label{boundhatB}
\|(\widehat{B}+
\mathcal{E}_{\widehat{B}})^{-1}\|=\|\widehat{B}^{-1}
+O(\mathcal{E}_{\widehat{B}})\|\rightarrow\|\widehat{B}^{-1}\|\leq \|B^{-1}\|
\end{equation}
is uniformly bounded independent of $\theta_1$.
Since $\|\widehat{A}\|$ is always bounded from above as $\|\widehat{D}\|\leq\|D\|$
and $\|\widehat{K}\|\leq\|K\|$, it follows that
$\|\widehat{B}^{-1}\widehat{A}\|\leq\|\widehat{B}^{-1}\|\|\widehat{A}\|$
is uniformly bounded independent of $\theta_1$. As a result, for
$\theta_1\rightarrow 0$, since $\widehat{A}+\mathcal{E}_{\widehat{A}}
\rightarrow \widehat{A}$, it follows from \eqref{boundhatB}
and Theorem~\ref{thm:conv_thm1} that
$$
\|(\widehat{B}+\mathcal{E}_{\widehat{B}})^{-1}
(\widehat{A}+\mathcal{E}_{\widehat{A}})\|\leq
\|(\widehat{B}+\mathcal{E}_{\widehat{B}})^{-1}\|
\|(\widehat{A}+\mathcal{E}_{\widehat{A}})\|
$$
is uniformly bounded independently of $\theta_1$.

Finally, from Theorem~\ref{thm:conv_thm1} and $(\widehat{B}+
\mathcal{E}_{\widehat{B}})^{-1}=\widehat{B}^{-1}
+O(\mathcal{E}_{\widehat{B}})$, it is easily justified
that
$$
\|\widehat{B}^{-1}\widehat{A}-
(\widehat{B}+\mathcal{E}_{\widehat{B}})^{-1}
(\widehat{A}+\mathcal{E}_{\widehat{A}})\|=O(\sin\theta_1).
$$

Based on Elsner's theorem, we have the following result, which, together
with the above discussions, proves the global unconditional convergence
of  Ritz values when $\theta_1\rightarrow 0$.

\begin{theorem}\label{cor:qritzvalue}
Assume that $\theta_1$ is small enough to make $\widehat{B}+
\mathcal{E}_{\widehat{B}}$ nonsingular. There is a  Ritz value $\mu_1$ such
that
\begin{equation}\label{boundqritzvalue}
|\mu_1-\lambda_1|\leq (\|\widehat{B}^{-1}\widehat{A}\|
+\|(\widehat{B}+\mathcal{E}_{\widehat{B}})^{-1}
(\widehat{A}+\mathcal{E}_{\widehat{A}})\|)^{1-\frac{1}{2m}}
\|\widehat{B}^{-1}\widehat{A}-
(\widehat{B}+\mathcal{E}_{\widehat{B}})^{-1}
(\widehat{A}+\mathcal{E}_{\widehat{A}})\|^{\frac{1}{2m}}.
\end{equation}
\end{theorem}

The theorem indicates that as $\theta_1\rightarrow 0$ there is always a  Ritz
value $\mu_1\rightarrow \lambda_1$ unconditionally.
We should comment that bound \eqref{boundqritzvalue} will in general be a too
pessimistic overestimate and be for the worst case.
If, as usually happens in practice, the condition
number of $\lambda_1$ as an eigenvalue of $(\widehat{B}+
\mathcal{E}_{\widehat{B}})^{-1} (\widehat{A}+\mathcal{E}_{\widehat{A}})$ is
bounded, the convergence will be linear in $\theta_1$, much
better than that predicted by bound \eqref{boundqritzvalue}.

Next, we analyze the convergence of the corresponding  Ritz vector $\tilde{x}_1$.
Based on decomposition \eqref{eq:RR_AB}, we can establish the following result,
which is an analogue of Theorem 3.1 in \cite{ipsen00}
for the standard linear eigenvalue problem. The result will be used when
we prove the unconditional convergence of refined  Ritz vectors to be
introduced in the next section.

\begin{lemma} \label{thm:conv_thm2}
      Let $( \mu_1, \tilde{v}_1  )$ with $\| \tilde{v}_1 \| = 1$ be an
      approximation to $\left( \lambda_1, v_1 \right)$ of the matrix pair
      $( A,B )$ with
      $\| v_1 \| = 1$.
      Let
      \begin{align}
             r = A \tilde{v}_1 - \mu_1 B \tilde{v}_1  \label{eq:residual_AB}
      \end{align}
      be the residual of $(\mu_1, \tilde{v}_1 )$, and define
      ${\rm sep}(\mu_1, ( L,N )) := \| ( L - \mu_1 N)^{-1} \|^{-1}$. Then
      \begin{eqnarray}
            \sin \angle ( v_1, \tilde{v}_1) \leq \frac{\| r \|}{{\rm sep}(\mu_1,
            ( L,N ))}. \label{eq:sin_sep}
      \end{eqnarray}
\end{lemma}
\begin{proof}
     From \eqref{eq:RR_AB}, pre-multiplying \eqref{eq:residual_AB} by $Y^H$ leads to
     \begin{align*}
          Y^H r =& Y^H \left( \alpha y_1 v_1^H + y_1 s^H X^H + Y L X^H \right)
          \tilde{v}_1\\
          &- \mu_1 Y^H \left( \beta y_1 v_1^H + y_1 t^H X^H +
          Y N X^H \right) \tilde{v}_1 \\
                =& \left( L - \mu_1 N \right) X^H \tilde{v}_1.
     \end{align*}
     Therefore, it follows from  $\| X^H \tilde{v}_1 \| =
     \sin \angle ( v_1, \tilde{v}_1)$
     that \eqref{eq:sin_sep} holds.
\end{proof}

In terms of the a posteriori computable residual $r$,
Theorem~\ref{thm:conv_thm2} establishes the relationship
between the eigenvector $v_1$ and its approximation $\tilde{v}_1$
for the generalized eigenvalue problem \eqref{eq:1stGEP}.

Let $(\mu_1, \widetilde{x}_1)$ be the  Ritz pair
approximating the desired the desired eigenpair $(\lambda_1,x_1)$ of $(M,D,K)$,
where $\widetilde{x}_1=Q\hat{x}_1$ and $(\mu_1,\hat{x}_1)$
with $\| \hat{x}_1 \| = 1$ is the eigenpair
of $( \widehat{M}, \widehat{D}, \widehat{K} )$.
In terms of $\theta_1$, we attempt to derive one of our main results,
an a priori bound for the  Ritz vector $\hat{x}_1$ as
an approximation to the eigenvector $x_1$.
Note that $\mu_1$ is an eigenvalue of $( \widehat{A}, \widehat{B} )$ and $\hat{v}_1
\equiv \left[ \begin{array}{c} \mu_{1} \hat{x}_1 \\ \hat{x}_1 \end{array}
\right]/\sqrt{1+| \mu_1|^2}$
is its corresponding normalized eigenvector. Similar to \eqref{eq:RR_AB}, there
are unitary matrices $[\hat{v}_1,\ \widehat{X}]$ and $[\hat{y}_1, \ \widehat{Y}]
\in \mathcal{C}^{2m\times 2m}$
with $\widehat{v}_1,\widehat{y}_1 \in \mathcal{C}^{2m}$ such that
\begin{align}
     \begin{bmatrix}
           \hat{y}_1^H \\ \widehat{Y}^H
     \end{bmatrix} \widehat{A} \begin{bmatrix}
           \hat{v}_1 & \widehat{X}
     \end{bmatrix} = \begin{bmatrix}
           \hat{\alpha} & \hat{s}^H \\ 0 & \widehat{L}
     \end{bmatrix}, \quad \begin{bmatrix}
           \hat{y}_1^H \\ \widehat{Y}^H
     \end{bmatrix} \widehat{B} \begin{bmatrix}
           \hat{v}_1 & \widehat{X}
     \end{bmatrix} = \begin{bmatrix}
           \hat{\beta} & \hat{t}^H \\ 0 & \widehat{N}
     \end{bmatrix}, \label{eq:RR_hatAB}
\end{align}
where $\widehat{L}$, $\widehat{N} \in \mathcal{C}^{(2m-1) \times (2m-1)}$
and  $\mu_1 = \hat{\alpha}\hat{\beta}^{-1}$.
Under the only hypothesis that $\sin\theta_1\rightarrow 0$, it is possible that
there is an eigenvalue of $( \widehat{L}, \widehat{N} )$
that could be arbitrarily near or even equal to $\mu_1$. For a multiple and
derogatory $\mu_1$, that is, $\mu_1$ has
more than one trivial or nontrivial Jordan blocks,
there are more than one $\tilde{x}_1 = Q \hat{x}_1$
to approximate the unique eigenvector $x_1$ of $( M,D,K )$.
If $\mu_1$ is near an eigenvalue of $( \widehat{L}, \widehat{N} )$,
we will get a unique $\tilde{x}_1$, but there is no guarantee
that it converges to $x_1$. It leads us to postulate that
$\tilde{x}_1$ will converge provided that
${\rm sep}(\lambda_1,( \widehat{L}, \widehat{N} ))$ is uniformly away from zero
independent of $\theta_1$, i.e., ${\rm sep}(\lambda_1,( \widehat{L}, \widehat{N} ))>c$
with $c$ a positive constant independent of $\theta_1$.
We will, quantitatively, show that it is indeed the case.
Before proceeding, we need the following lemma.

\begin{lemma} \label{lem:conv_lem3}
     Let $u = \left[ \begin{array}{c} u_2 \\ u_1 \end{array} \right]$
     and $\tilde{u} = \left[ \begin{array}{c} \tilde{u}_2 \\ \tilde{u}_1
     \end{array} \right]$ where $u_i, \tilde{u}_i \in \mathcal{C}^n$ for $i = 1, 2$
     and $\| u_1 \| = \| \tilde{u}_1 \| = 1$.
     Then
     $$
      \sin \angle (u_1, \tilde{u}_1) \leq  \min\{\| u \|,\|\tilde{u}\|\}
      \sin \angle (u, \tilde{u}).
     $$
\end{lemma}
\begin{proof}
Since $\|u_1\|=1$, from the definition of $\sin\angle (u,\tilde{u})$,
we have
\begin{align*}
\sin^2\angle (u,\tilde{u})=&\min_{\alpha}\left\|\frac{u}{\|u\|}-\alpha \tilde u\right\|^2\\
=&\min_{\alpha}\left(\left\|\frac{u_1}{\|u\|}-\alpha \tilde{u}_1\right\|^2+
\left\|\frac{u_2}{\|u\|}-\alpha \tilde{u}_2\right\|^2\right)\\
\geq&\min_{\alpha}\left\|\frac{u_1}{\|u\|}-\alpha \tilde{u}_1\right\|^2\\
=&\frac{1}{\|u\|^2}\min_{\alpha}\|u_1-\alpha\tilde{u}_1\|^2\\
=&\frac{1}{\|u\|^2}\sin^2\angle (u_1,\tilde{u}_1).
\end{align*}
In the same way, we can also prove that
$$
\sin \angle (u_1, \tilde{u}_1) \leq  \|\tilde{u}\| \sin \angle (u, \tilde{u}).
$$
Therefore, the assertion holds.
\end{proof}

\begin{theorem} \label{thm:conv_thm3}
     Let $( \widehat{A}, \widehat{B} )$ be defined in \eqref{eq:mtx_hat_AB} and
     it have decomposition \eqref{eq:RR_hatAB}. Suppose that the  Ritz pair
     $(\mu_1,\tilde{x}_1)$  is used to approximate the desired eigenpair
     $(\lambda_1,x_1)$ with $\| \tilde{x}_1 \| = \| x_1 \| = 1$.
     If {\rm sep}$(\lambda_1, ( \widehat{L}, \widehat{N} )) > 0$, then
     \begin{equation}\label{boundqritzvector}
           \sin \angle ( x_1, \tilde{x}_1) \leq  \sin \theta_1 +
           \frac{ | \lambda_1 |^2 m_0 + |\lambda_1| d_0 + k_0}
           {{\rm sep}(\lambda_1, ( \widehat{L}, \widehat{N} ))} \tan \theta_1,
     \end{equation}
     where $m_0,\ d_0$ and $k_0$ are defined in \eqref{m0d0k0}.
\end{theorem}
\begin{proof}
      By the triangle inequality we have
      \begin{eqnarray}
            \angle( x_1, \tilde{x}_1) \leq \angle(x_1, Q Q^H x_1) +
            \angle( QQ^H x_1, \tilde{x}_1).
             \label{eq:pf_conv_thm4_1}
      \end{eqnarray}
      From \eqref{eq:theta} and \eqref{eq:q1_norm}, we have
      \begin{eqnarray}
            \cos \angle( x_1, Q Q^H x_1) = \frac{| x_1^H Q Q^H x_1|}{\| Q Q^H x_1 \|}
            = \| Q^H x_1 \| =  \cos \theta_1. \label{eq:pf_conv_thm4_0}
      \end{eqnarray}
      Let $\hat{q}_1 = \frac{Q^H x_1}{\| Q^H x_1\|}$. From \eqref{eq:pf_conv_thm4_1}
      and \eqref{eq:pf_conv_thm4_0} we get
      \begin{align}
             \sin \angle(x_1, \tilde{x}_1) \leq& \sin \theta_1 +
             \sin \angle( QQ^H x_1, \tilde{x}_1) = \sin \theta_1 +
             \sin \angle ( Q \hat{q}_1, Q \hat{x}_1) \nonumber \\
             =& \sin \theta_1 + \sin \angle ( \hat{x}_1, \hat{q}_1).
             \label{eq:pf_conv_thm4_2}
      \end{align}
      From \eqref{eq:mtx_hat_AB}, it is easily seen that $( \mu_1, \hat{v}_1 \equiv
      \left[ \begin{array}{c} \mu_1 \hat{x}_1 \\ \hat{x}_1 \end{array} \right] )$ is
      an eigenpair of $( \widehat{A}, \widehat{B})$. So we can regard $(\lambda_1,
      \hat{q} \equiv
      \left[ \begin{array}{c} \lambda_1 \hat{q}_1 \\  \hat{q}_1 \end{array} \right] )$
      as an approximation of $(\mu_1, \hat{v}_1)$. Then the residual of
      $(\lambda_1, \hat{q})$ as an approximate eigenpair of $(\widehat{A},\widehat{B})$
      is
      \begin{align*}
            \hat{r} &= \left[ \begin{array}{cc}
                   - \widehat{D} & -\widehat{K} \\ I_m & 0
            \end{array} \right] \left[ \begin{array}{c}
                   \lambda_1 \hat{q}_1 \\  \hat{q}_1
           \end{array} \right] - \lambda_1 \left[ \begin{array}{cc}
                   \widehat{M} & 0 \\ 0 & I_{m}
           \end{array} \right] \left[ \begin{array}{c}
                   \lambda_1 \hat{q}_1 \\ \hat{q}_1
           \end{array} \right] \nonumber \\
           &= \left[ \begin{array}{c}
                  -\left( \lambda_1^2 \widehat{M} + \lambda_1 \widehat{D} +
                  \widehat{K} \right) \hat{q}_1 \\
                  0
           \end{array} \right] \equiv \left[ \begin{array}{c}
                  -\hat{r}_1 \\ 0
           \end{array} \right].
      \end{align*}
      By \eqref{eq:pf_conv_thm1_3} in the proof of Theorem~\ref{thm:conv_thm1} we have
       \begin{eqnarray}
            \frac{\| \hat{r} \| }{ \| \hat{q} \|} = \frac{\| \hat{r}_1 \|}{\|\hat{q}\|}
            \leq \frac{ | \lambda_1|^2 m_0 + | \lambda_1 | d_0 +
            k_0}{\|\hat{q}\|} \tan \theta_1. \label{eq:pf_conv_thm4_3}
       \end{eqnarray}
       From Lemma~\ref{lem:conv_lem3}, Theorem~\ref{thm:conv_thm2} and
       \eqref{eq:pf_conv_thm4_3}, inequality \eqref{eq:pf_conv_thm4_2} satisfies
      \begin{align*}
            \sin \angle(x_1, \tilde{x}_1) &\leq \sin \theta_1 +
            \sin \angle(\hat{x}_1, \hat{q}_1) \\
            &\leq \sin \theta_1 + \|\hat{q}\| \sin \angle ( \hat{v}_1, \hat{q} ) \\
            &\leq \sin \theta_1 + \|\hat{q}\|  \frac{\| \hat{r} \| / \| \hat{q} \| }
            {\mbox{sep}(\lambda_1, ( \widehat{L}, \widehat{N} ))} \\
            &\leq \sin \theta_1 +  \frac{  | \lambda_1 |^2 m_0 + |\lambda_1| d_0 + k_0}
            {\mbox{sep}(\lambda_1, ( \widehat{L}, \widehat{N} ))} \tan \theta_1.
      \end{align*}
\end{proof}

From Theorem~\ref{thm:conv_thm3} we see that sep$(\lambda_1,
( \widehat{L}, \widehat{N} ))>0$ uniformly
is a sufficient condition for the convergence of
the  Ritz vector $\tilde{x}_1$. Furthermore, from
Lemma~\ref{thm:conv_thm1}, since the Ritz value $\mu_1$ approaches the
eigenvalue $\lambda_1$ as $\theta_1 \to 0$, by the
continuity argument we have sep$(\mu_1, ( \widehat{L}, \widehat{N} ))
\to$sep$(\lambda_1, ( \widehat{L}, \widehat{N} ))$.
However, as we have argued above, sep$(\mu_1, ( \widehat{L}, \widehat{N} ))$
can be arbitrarily small (and even be exactly zero) when $\mu_1$ is arbitrarily
near other eigenvalues (or is associated with a multiple eigenvalue)
of $( \widehat{L},\widehat{N} )$. Consequently, while the  Ritz value
converges unconditionally once $\theta_1 \to 0$, the corresponding  Ritz
vector may fail to converge or may converge very slowly or irregularly.

In the following, we give an example to illustrate that the  Ritz vector fails
to converge to the desired eigenvector.
\begin{example}\label{Example}
    Consider QEP \eqref{eq:QEP} with
    \begin{eqnarray*}
        M = \left[ \begin{array}{ccc}
           1 & 1 & 0 \\ 1 & 2 & 1 \\ 0 & 1 & 2
        \end{array} \right], \ D = \left[ \begin{array}{rrr}
           -5.5 & -5 & 0 \\ -5 & -11 & -3 \\ 0 & -3 & -4
        \end{array} \right], \ K = \left[ \begin{array}{rrr}
           6 & 6 & 0 \\ 6 & 9 & 2 \\ 0 & 2 & 2
        \end{array} \right].
    \end{eqnarray*}
    It is easy to see that $M$ and $K$ are symmetric positive definite and
    $(1, [0, 0, 1]^T)$ is an eigenpair of the QEP.
\end{example}
Suppose that we have come up with an orthonormal basis
\begin{eqnarray*}
    Q = \left[ \begin{array}{cc}
       0 & \frac{8}{\sqrt{73}} \\ 0 & -\frac{3}{\sqrt{73}} \\ 1 & 0
    \end{array} \right].
\end{eqnarray*}
Then we have $\sin\theta_1=0$ exactly, and the projected matrices are
\begin{align*}
    \widehat{M} &= Q^H M Q = \left[ \begin{array}{rr}
        2 & - \frac{3}{\sqrt{73}} \\ - \frac{3}{\sqrt{73}} & \frac{34}{73}
    \end{array} \right], \\
    \widehat{D} &= Q^H D Q = \left[ \begin{array}{rr}
        -4 & \frac{9}{\sqrt{73}} \\ \frac{9}{\sqrt{73}} & - \frac{211}{73}
    \end{array} \right], \\
    \widehat{K} &= Q^H K Q = \left[ \begin{array}{rr}
       2 & -\frac{6}{\sqrt{73}} \\ -\frac{6}{\sqrt{73}} & \frac{177}{73}
    \end{array} \right],
\end{align*}
from which it follows that
\begin{eqnarray*}
     \widehat{M} + \widehat{D} + \widehat{K} = 0.
\end{eqnarray*}
Since $\widehat{M} + \widehat{D} + \widehat{K}$ is zero, any nonzero vector
$\hat{x}_1$ with $\| \hat{x}_1 \| = 1$ is an eigenvector of $( \widehat{M},
\widehat{D},\widehat{K} )$ corresponding to the double eigenvalue one,
a  Ritz value equal to the desired eigenvalue exactly.
However, the  Rayleigh-Ritz method itself cannot tell us how
to pick up a suitable $\hat{x}_1$. In practice, we might well take $\hat{x}_1 =
[ 1/\sqrt{2}, 1/\sqrt{2}]^T$ and then the approximate eigenvector becomes
$[ 4\sqrt{2}/\sqrt{73}, \ -3/\sqrt{146},\ 1/\sqrt{2}]^T$,
which has no accuracy as an approximation of the desired eigenvector $[0,0,1]^T$ and
is completely wrong. Thus the method can fail even though the projection subspace
span$\{Q\}$ contains the desired eigenvector exactly.

In practice, we would not expect ${\rm span}\{Q\}$ to contain $x_1$ exactly.
Let us investigate the case that ${\rm span}\{Q\}$ contains an enough
accurate approximation to $x_1$, i.e., $\sin\theta_1$ is very small.
We perturb $Q$ by a matrix generated randomly
in a normal distribution by
$10^{-12}\times {\sf randn(3,2)}$ whose 2-norm is $2.2\times 10^{-12}$,
and the resulting
$$
\sin\theta_1=1.7\times 10^{-12}.
$$
The orthonormalized
$$
Q:=Q(Q^HQ)^{-1/2}=\left[\begin{array}{rr}
   -0.000000000001074  & 0.936329177568703\\
   -0.000000000001425  &-0.351123441589302\\
   1.000000000000000  &0.000000000000506\\
\end{array}
\right]
$$
and
$$
\widehat{M}=\left[\begin{array}{rr}
 1.999999999997149 &  -0.351123441589253\\
  -0.351123441589253&  0.465753424656353\\
\end{array}
\right],
$$
$$
\widehat{D}=\left[\begin{array}{rr}
 -3.999999999991449   & 1.053370324770698\\
  1.053370324770698& -2.890410958899234\\
\end{array}
\right],
$$
$$
\widehat{K}=\left[\begin{array}{rr}
1.999999999997149 &-0.351123441589253\\
-0.351123441589253 & 0.465753424656353\\
\end{array}
\right].
$$
We use the Matlab function {\sf polyeig.m} to solve the projected QEP, and the computed
$\mu_1=1.000000000009369$ and the associated eigenvector
$$
\hat{x}_1=[ 0.999982126253304,
-0.005978894038382]^T.
$$
So the  Ritz vector
$$
\tilde{x}_1=Q\hat{x}_1=[
 -0.005598212938803,
 0.002099329850230,
 0.999982126253300]^T
$$
and
$$
\sin\angle(x_1,\tilde{x}_1)\approx  0.005979,
$$
at least nine orders bigger than
$\sin\theta_1$! so $\tilde{x}_1$ is a very poor approximation to $x_1$ for the
given accurate subspace ${\rm span}\{Q\}$. It is also justified that the residual
norm of the  Ritz pair $(\mu_1,\tilde{x}_1)$ is
$$
\|(\mu_1^2M+\mu_1D+K)\tilde{x}_1\|
\approx  0.011958.
$$
The poor accuracy of $\tilde{x}_1$ is due to the fact that
there is another  Ritz value $\mu=1.000000000010143$ that is very
near to $\mu_1$, so that ${\rm sep}(\lambda_1, ( \widehat{L}, \widehat{N} ))$
in \eqref{boundqritzvector} is tiny.

\section{Convergence of refined  Ritz vectors} \label{sec:conv_refined_q_Ritz}

As we have seen in Section~\ref{sec:conv_q_Ritz}, the  Ritz vector may fail
to converge or converges very slowly. Since the  Ritz value is known to converge to
the simple eigenvalue $\lambda_1$ when $\sin \theta_1 \to 0$, this suggests us to deal
with non-converging  Ritz vector by retaining the  Ritz value but replacing
the  Ritz vector with a unit length vector $\tilde{z}_1 \in
\mbox{span}\{ Q \}$ with a suitably
small residual. Naturally, for a given  Ritz value $\mu_1$ we
construct $\tilde{z}_1 = Q \hat{z}_1$, where the unit length
$\hat{z}_1$ is required to be the optimal solution
\begin{eqnarray}
      \hat{z}_1 = \mbox{arg} \min_{\| z \| = 1} \left\| \left( \mu_1^2 M +
      \mu_1 D + K \right) Q z \right\|. \label{eq:optimal_RV}
\end{eqnarray}
The vector $\tilde{z}_1=Q \hat{z}_{1}$ is called a refined  Ritz
vector of $( M,D,K )$ corresponding to $\mu_{1}$ with respect to
${\rm span}\{Q\}$. Obviously, $\hat{z}_1$ is the right singular vector
of the $n\times m$ rectangular matrix
$\left( \mu_1^2 M +\mu_1 D + K \right) Q$ associated with its smallest
singular value. We can compute $\hat{z}_1$ reliably by a standard SVD algorithm or
generally cheaper but still numerically stable cross-product based SVD algorithms;
see \cite{Jia00,JiaSun11} and also \cite{Stewart_MaxAlg_Eig_01}.
For a detailed round-off error analysis on the latter ones,
we refer to \cite{Jia06}.

Before establishing the convergence of the refined  Ritz vector $\tilde{z}_1$,
we need two lemmas.

\begin{lemma} \label{lem:2}
For $W$ defined in \eqref{eq:mtx_W}, let $(\lambda_1, x_1)$ with
$\| x_1 \| = 1$ be the desired eigenpair of
      $( M,D,K )$ and $v_1 = \left[ \begin{array}{c} \lambda_1 x_1 \\
      x_1 \end{array} \right] / \sqrt{1 + | \lambda_1 |^2 }$.
Then it holds that
\begin{equation}
\sin\angle(v_1,{\rm span}\{W\})
=\sin\theta_1.
\end{equation}
\end{lemma}

\begin{proof}
By \eqref{eq:mtx_W} and the definition of $\sin\theta_1$, we have
\begin{align*}
& \hspace{4mm} \sin^2\angle\left(v_1,{\rm span}\left\{ W \right\}\right) \\
&= \frac{1}{1+\mid\lambda_1\mid^2}\min_{u, v\in{\rm span}\{Q\}}
\left\|\left[\begin{array}{c}
\lambda_1 x_1\\
x_1
\end{array}\right]-\left[\begin{array}{c}
u\\
v
\end{array}
\right]\right\|^2\\
&= \frac{1}{1+\mid\lambda_1\mid^2}\min_{u, v\in {\rm span\{Q\}}}(\|\lambda_1
x_1-u\|^2+\|x_1-v\|^2)\\
&= \frac{\mid\lambda_1\mid^2}{1+\mid\lambda_1\mid^2}\min_{u\in
{\rm span\{Q\}}}\|x_1-u\|^2 
  +\frac{1}{1+\mid\lambda_1\mid^2}\min_{v\in {\rm span\{Q\}}}\|x_1-v\|^2\\
&= \frac{\mid\lambda_1\mid^2}{1+\mid\lambda_1\mid^2}\sin^2\theta_1+
\frac{1}{1+\mid\lambda_1\mid^2}\sin^2\theta_1\\
&= \sin^2\theta_1.
\end{align*}
\end{proof}

\begin{lemma}\label{lem:3}
Let $(A, B)$ be defined in \eqref{eq:mtx_A_B_1stGEP}. It holds that
\begin{equation}\label{qepgen}
\min_{\|z\|=1}\left\|(A-\mu_1 B)\left[\begin{array}{c}
\mu_1 Qz\\
Qz
\end{array}\right]\right\|=\sqrt{1+\mid\mu_1\mid^2}\min_{\| z \| = 1}
\left\| \left( \mu_1^2 M +
      \mu_1 D + K \right) Q z \right\|
\end{equation}
and the minimum is attained at $\hat{z}_1$.
\end{lemma}

\begin{proof}
Without the minimizations, for any $m$ dimensional vector $z$,
it is direct to verify that the two hand sides are equal. So
the assertion holds.
\end{proof}

\begin{theorem}
      Let $\mu_1$ be the  Ritz value of $( {M},  {D},  {K} )$
      approximating the desired simple eigenvalue $\lambda_1$. Suppose
      {\rm sep}$(\mu_1, ( L, N )) > 0$, where $L, N$ are defined
      in \eqref{eq:RR_AB}.
Then we have
      \begin{equation}\label{rconv}
      \sin \angle(x_1, \tilde{z}_1)<\frac{\sqrt{1+|\lambda_1|^2}
      \left(|\lambda_1-\mu_1|(\|B\|+\|A-\mu_1 B\|)+\|A-\mu_1 B\|\sin\theta_1\right)}
{\cos\theta_1{\rm sep}(\mu_1, ( L, N ))}.
      \end{equation}
\end{theorem}
\begin{proof}
Let $v_1 = \left[ \begin{array}{c} \lambda_1 x_1 \\
      x_1 \end{array} \right] / \sqrt{1 + | \lambda_1 |^2 }$.
     From Lemma~\ref{lem:conv_lem3}, we have
     \begin{align*}
           \sin \angle(x_1, \tilde{z}_1) &\leq \sqrt{1 + | \mu_1 |^2}
           \sin \angle(\left[ \begin{array}{c} \lambda_1 x_1 \\
           x_1 \end{array} \right], \left[ \begin{array}{c} \mu_1 Q \hat{z}_1 \\
           Q \hat{z}_1 \end{array} \right]) \nonumber \\
           &=  \sqrt{1 + | \mu_1 |^2} \sin \angle(v_1, \tilde{z}),
     \end{align*}
     where $ \tilde{z} = \left[ \begin{array}{c} \tilde{z}_2 \\ \tilde{z}_1
     \end{array} \right] \equiv \left[ \begin{array}{c}
     \mu_1 Q \hat{z}_1 \\   Q \hat{z}_1 \end{array} \right] /
     \sqrt{1+|\mu_1|^2 }$.
Let $P_W$ be the orthogonal projector onto the subspace
${\rm span}\{W\}$, where $W = \mbox{diag}(Q,Q)$. Then
$$
P_Wv_1=\left[\begin{array}{c}
\lambda_1 QQ^H x_1\\
QQ^H x_1
\end{array}
\right].
$$
Therefore, we get
$$
\|Q^Hx_1\|^{-1} \left( P_Wv_1-\left[\begin{array}{c}
(\lambda_1-\mu_1) QQ^H x_1\\
0
\end{array}
\right]\right)=\left[\begin{array}{c}
\mu_1 Q\frac{Q^H x_1}{\|Q^H x_1\|}\\
Q\frac{Q^H x_1}{\|Q^H x_1\|}
\end{array}
\right]:=\hat{v}_1,
$$
which is an approximate eigenvector of the desired form in
the left-hand side of \eqref{qepgen}
and $\frac{Q^Hx_1}{\|Q^H x_1\|}$ is a minimizer candidate for \eqref{qepgen}.
Define
$$
f=(I_n-P_W)v_1+f_2
$$
with
$$
f_2=\left[\begin{array}{c}
(\lambda_1-\mu_1) QQ^H x_1\\
0
\end{array}
\right].
$$
Then from $\cos\theta_1=\|Q^Hx_1\|$ we have
$$
\frac{\|f_2\|}{\cos\theta_1}\leq |\lambda_1-\mu_1|.
$$
From Lemma~\ref{lem:2} we get $\|(I_n-P_W)v_1\|=\sqrt{1+|\lambda_1|^2}\sin\theta_1$.
Therefore, we obtain
\begin{align*}
(A-\mu_1 B)\hat{v}_1&= \frac{(A-\mu_1 B)(P_W v_1-f_2)}{\cos\theta_1}\\
&= \frac{(A-\mu_1 B)(v_1-f)}{\cos\theta_1}\\
&= \frac{(\lambda_1-\mu_1)B v_1-(A-\mu_1 B)((I_n-P_W)v_1+f_2)}{\cos\theta_1}.
\end{align*}
Taking the norms gives
$$
\|(A-\mu_1 B)\hat{v}_1\|\leq\frac{\sqrt{1+|\lambda_1|^2}(|\lambda_1-
\mu_1|\|B\|+\|A-\mu_1 B\|\sin\theta_1)}
{\cos\theta_1}+|\lambda_1-\mu_1|\|A-\mu_1 B\|.
$$
From Lemma~\ref{lem:3}, by the optimality property of $\tilde{z}$ we have
{\small
\begin{align*}
 \frac{\|(A-\mu_1 B)\tilde{z}\|}{\sqrt{1+|\mu_1|^2}} &\leq
\frac{\|(A-\mu_1 B)\hat{v}_1\|}{\sqrt{1+|\mu_1|^2}}\\
&\leq
\frac{\sqrt{1+|\lambda_1|^2}(|\lambda_1-\mu_1|\|B\|+\|A-\mu_1 B\|\sin\theta_1)}
{\sqrt{1+|\mu_1|^2}\cos\theta_1}+\frac{|\lambda_1-\mu_1|\|A-\mu_1 B\|}
{\sqrt{1+|\mu_1^2|}}.
\end{align*}}
Since $\frac{\|(A-\mu_1 B)\tilde{z}\|}{\sqrt{1+|\mu_1|^2}}$ is a residual norm,
it is direct from
Theorem~\ref{thm:conv_thm2} that
$$
\sin\angle(v_1,\tilde{z})\leq \frac{\|(A-\mu_1 B)\tilde{z}\|}
{\sqrt{1+|\mu_1|^2}{\rm sep}(\mu_1, ( L, N ))}.
$$
Therefore, it holds from Lemma~\ref{lem:conv_lem3} that
{\small
\begin{align*}
\sin\angle(x_1,\tilde{z}_1)&\leq\sqrt{1+|\mu_1|^2}\sin\angle(v_1,\tilde{z})\\
&\leq\frac{\sqrt{1+|\lambda_1|^2}(|\lambda_1-\mu_1|\|B\|+\|A-\mu_1 B\|\sin\theta_1)}
{\cos\theta_1{\rm sep}(\mu_1, ( L, N ))}+
\frac{|\lambda_1-\mu_1|\|A-\mu_1 B\|}{{\rm sep}(\mu_1, ( L, N ))}\\
&<\frac{\sqrt{1+|\lambda_1|^2}
      (|\lambda_1-\mu_1|\left(\|B\|+\|A-\mu_1 B\|)+\|A-\mu_1 B\|\sin\theta_1\right)}
{\cos\theta_1{\rm sep}(\mu_1, ( L, N ))},
\end{align*}}
which proves (\ref{rconv}).
\end{proof}

Since $\mu_1$ is shown, as Corollary~\ref{cor:qritzvalue} indicates,
to converge to $\lambda_1$ as $\theta_1 \to 0$, we have sep$(\mu_1,
( L, N )) \to \mbox{sep}(\lambda_1, ( L, N ))$,
a positive constant independent of $\theta_1$,
provided that $\lambda_1$ is a simple eigenvalue of $( M,D,K )$.
So the refined  Ritz vector $\tilde{z}_1$ converges to $x_1$
once $\sin\theta_1 \to 0$.

We mention that Hochstenbach and Sleijpen \cite{hochsleij08} proposed a
refined Rayleigh--Ritz method for the polynomial eigenvalue problem
and derived an a priori bound for the residual norm of the refined
Ritz pair as the approximate eigenpair of the problem without invoking any
linearization; see Theorem 5.1 there.

We continue Example ~\ref{Example} to show considerable merits of refined
Ritz vectors. For the case that $x_1$ lies in ${\rm span}\{Q\}$ exactly,
recall that $\mu_1=\lambda_1$ exactly. It is easy to verify that the
smallest singular value of the matrix
$(\mu_1^2M+\mu_1D+K)Q$ is both exactly zero and simple, the optimal solution
$\hat{z}_1=[1,0]^T$ in \eqref{eq:optimal_RV} and
the refined  Ritz vector $\tilde{z}_1=Q\hat{z}_1=x_1$, exactly the desired
eigenvector! So in contrast to the  Ritz vector, the refined  Ritz vector can
pick up the desired eigenvector perfectly.

For the case that ${\rm span}\{Q\}$ is perturbed in the way described
in Example ~\ref{Example}, the optimal solution in \eqref{eq:optimal_RV} is
$$
\hat{z}_1=[1.000000000000000,0.000000000006175]^T
$$
and the refined  Ritz vector
$$
\tilde{z}_1=[
0.000000000004708,
  -0.000000000003593,
   1.000000000000000
]^T.
$$
So
$$
\sin \angle(x_1, \tilde{z}_1)=5.9\times 10^{-12},
$$
which is almost as small as $\sin\theta_1=1.7\times 10^{-12}$
and much more accurate than the
corresponding  Ritz vector $\tilde{x}_1$. Meanwhile, the computed
residual norm of the refined approximate eigenpair $(\mu_1,\tilde{z}_1)$
is
$$
\|(\mu_1^2M+\mu_1D+K)\tilde{z}_1\|=1.3\times 10^{-13},
$$
eleven orders smaller than that of the  Ritz pair $(\mu_1,\tilde{x}_1)$.

\section{Conclusions} \label{sec:conclusion}

Theoretically, we have proved that
there exists a  Ritz value of $( M,D,K )$ that unconditionally converges
to the desired eigenvalue when the angle between the subspace
span$\{ Q \}$ and the desired eigenvector tends to zero. However,
the associated  Ritz vector only converges conditionally. To this end, we
have proposed the refined  Ritz vector that is guaranteed to converge
unconditionally. We have presented some examples to demonstrate
our theory.

The purpose of this paper is not to present efficient and reliable
eigensolvers for QEPs, but rather to establish a general convergence
theory of the  Rayleigh-Ritz method and to show the unconditional convergence
of  Ritz values and refined  Ritz vectors and the conditional convergence
of  Ritz vectors. Refined  Ritz vectors may become a very valuable component
and make great improvement in flexible eigensolvers for QEPs. Numerical
experiments in \cite{JiaSun11} have shown that one can gain very much by
replacing  Ritz vectors by refined  Ritz vectors in second-order Arnoldi
type methods and their implicitly restarted algorithms.

\begin{acknowledgements}
We thank the editor Professor Michiel Hochstenbach and the
referee very much for their valuable suggestions and
comments that made us improve the presentation of the paper very substantially.
\end{acknowledgements}

\end{document}